\documentclass{amsart}

\usepackage{amsmath}
\usepackage{amsthm}
\usepackage{amsfonts}
\usepackage{dsfont}
\usepackage{enumitem}

\newcommand{\reals}{\mathbb{R}}
\newcommand{\rationals}{\mathbb{Q}}
\newcommand{\complex}{\mathbb{C}}
\newcommand{\naturals}{\mathbb{N}}

\newcommand{\bracketb}[1]{\Big[#1\Big]}

\newcommand{\angles}[1]{\left\langle #1 \right\rangle}

\newcommand{\paraa}[1]{\big(#1\big)}
\newcommand{\parab}[1]{\Big(#1\Big)}

\newcommand{\spacearound}[1]{\quad#1\quad}

\renewcommand{\implies}{\spacearound{\Rightarrow}}

\newtheorem{theorem}{Theorem}[section]

\newtheorem{lemma}[theorem]{Lemma}
\newtheorem{proposition}[theorem]{Proposition}

\theoremstyle{definition}
\newtheorem{definition}[theorem]{Definition}
\theoremstyle{remark}
\newtheorem{remark}[theorem]{Remark}
\numberwithin{equation}{section}

\renewcommand{\mid}{\mathds{1}}
\newcommand{\A}{\mathcal{A}}

\renewcommand{\d}{\partial}
\newcommand{\dt}{\tilde{\partial}}
\renewcommand{\P}{\mathcal{P}}
\newcommand{\nablab}{\bar{\nabla}}

\newcommand{\Der}{\operatorname{Der}}
\newcommand{\nablasub}[1]{\nabla_{\!#1}}

\newcommand{\nablad}{\nablasub{d}}

\newcommand{\g}{\mathfrak{g}}
\newcommand{\gphi}{\g_{\varphi}}
\newcommand{\Mphi}{M_{\varphi}}

\newcommand{\half}{\frac{1}{2}}
\newcommand{\thalf}{\tfrac{1}{2}}

\newcommand{\Ws}{W^\ast}
\newcommand{\Zs}{Z^\ast}
\newcommand{\qb}{\bar{q}}

\newcommand{\Wsq}{|W|^2}
\newcommand{\Zsq}{|Z|^2}
\newcommand{\Wsqi}{|W|^{-2}}
\newcommand{\Zsqi}{|Z|^{-2}}
\newcommand{\xv}{\vec{x}}
\newcommand{\dup}[1]{\d^{(#1)}}
\newcommand{\dn}[1][n]{\dup{#1}}

\newcommand{\Ric}{\operatorname{Ric}}

\newcommand{\Ts}{T^\ast}

\newcommand{\Sfourt}{S^4_{\theta}}
\newcommand{\ZSfourt}{Z(\Sfourt)}
\newcommand{\TSfourt}{T\Sfourt}
\newcommand{\Sfourtloc}{S^4_{\theta,\operatorname{loc}}}
\newcommand{\hd}{h^{\delta}}
\newcommand{\td}{\tau_\delta}
\newcommand{\tfi}{(\mid-T^2)^{-1}}
\newcommand{\Ih}{\hat{I}}
\newcommand{\oneh}{\hat{1}}
\newcommand{\Zloc}{Z_{\operatorname{loc}}}
\newcommand{\tdloc}{\tau_{\delta,\operatorname{loc}}}

\title[]{On the Gauss-Chern-Bonnet theorem\\ for the noncommutative 4-sphere}
\author{Joakim Arnlind and Mitsuru Wilson}

\address[Joakim Arnlind]{Department of Mathematics\\
Link\"oping University\\
581 83 Link\"oping\\
Sweden}
\email{joakim.arnlind@liu.se}

\address[Mitsuru Wilson]{%
Middlesex College\\
University of Western Ontario\\
London, Ontario\\
N6A 5B7\\
Canada}
\email{mwils57@uwo.ca}



\begin{document}

\begin{abstract}
  We construct a differential calculus over the noncommutative
  4-sphere in the framework of pseudo-Riemannian calculi, and show
  that for every metric in a conformal class of perturbations of the
  round metric, there exists a unique metric and torsion-free
  connection. Furthermore, we find a localization of the projective
  module corresponding to the space of vector fields, which
  allows us to formulate a Gauss-Chern-Bonnet type theorem for the
  noncommutative 4-sphere.
\end{abstract}

\maketitle

\section{Introduction}

\noindent
Over the last years, there have been increasing interests in
understanding the curvature of noncommutative manifolds. Starting from
seminal work on the scalar curvature and Gauss-Bonnet type theorems
for the noncommutative torus
\cite{ct:gaussBonnet,fk:gaussBonnet,cm:modularCurvature} many
interesting papers that discuss different aspects of curvature in the
noncommutative setting have followed
\cite{r:leviCivita,ds:curved.torus.gb,fk:scalarCurvature,a:curvatureGeometric,fk:curvatureFourTori,lm:modularCurvature,l:modular.toric,ps:nc.levi.civita,ds:asymmetric,aw:curvature.three.sphere,esw:moyal.sphere}. Note
that these are only examples of recent progress in the area; several
authors have previously considered curvature in this context (see
e.g. \cite{cff:gravityncgeometry,dvmmm:onCurvature,m:nc.riemannian.spin,ac:ncgravitysolutions,bm:starCompatibleConnections,ahh:multilinear}).
Although connections on projective modules and their corresponding
curvatures are natural objects in noncommutative geometry, classical
objects that are built from the curvature tensor, like Ricci and
scalar curvature, do not always have straight-forward
analogues. Therefore, it is interesting to study as to what extent
such concepts are relevant for noncommutative geometry.

For Riemannian manifolds, the Gauss-Chern-Bonnet theorem provides an
important link between geometry and topology. It states that the
integral of the Pfaffian of the curvature form (of a closed even
dimensional manifold) is proportional to the Euler characteristic,
which is a topological invariant. For a two dimensional manifold, the
Pfaffian is simply the scalar curvature, which reduces the
Gauss-Chern-Bonnet theorem to the Gauss-Bonnet theorem. Therefore, to
understand similar theorems for two dimensional noncommutative
manifolds, one needs to find a proper definition of the scalar
curvature. For a Riemannian manifold, the asymptotic expansion of the
heat kernel contains information about the scalar curvature in one of
the coefficients. The expansion of the heat kernel makes sense even
for a noncommutative manifold, and the very same coefficient serves as
a definition of noncommutative scalar curvature.  For the
noncommutative torus, the scalar curvature corresponding to certain
perturbations of the flat metric has been computed, and it is possible
to show that a Gauss-Bonnet type theorem holds; i.e., the trace of the
scalar curvature is independent of the metric perturbation
\cite{ct:gaussBonnet,fk:gaussBonnet}.  However, for higher dimensional
manifolds, it is not clear how to define the analogue of the Pfaffian
of the curvature form in order to formulate the Gauss-Chern-Bonnet
theorem.

In this paper we construct a differential calculus over the
noncommutative 4-sphere, in the framework of pseudo-Riemannian calculi
\cite{aw:curvature.three.sphere}, and introduce a projective module in
close analogy with the space of vector fields on the classical
4-sphere. Moreover, via a suitable localization of the algebra, we
find a local trivialization of the projective module and prove
the existence of (unique) metric and torsion-free connections for a class of
perturbations of the round metric. Finally, we show that in this
particular case, there exists a naive analogue of the Pfaffian of the
curvature form, which allows us to prove a Gauss-Chern-Bonnet type
theorem for the noncommutative 4-sphere.

This paper is organized as follows: In
Section~\ref{sec:pseudo.riemannian.calculi} we briefly recall the
concept of a pseudo-Riemannian calculus, and
Section~\ref{sec:embedding} introduces a particular parametrization of
the classical 4-sphere. Sections \ref{sec:basic.properties} and
\ref{sec:real.metric.calculus} are devoted to the construction of a
real metric calculus over the noncommutative 4-sphere, and
Section~\ref{sec:local.algebra} discusses certain aspects of
localization. These results are then used in
Section~\ref{sec:pseudo.Riemannian} to construct a pseudo-Riemannian
calculus, giving metric and torsion-free connections for a class of
perturbed metrics. Finally, Section~\ref{sec:gcb} introduces a trace
for the noncommutative 4-sphere, and formulates a version of the
Gauss-Chern-Bonnet theorem.

\section{Preliminaries}\label{sec:preliminaries}

\subsection{Pseudo-Riemannian calculi}\label{sec:pseudo.riemannian.calculi}

\noindent
Let us briefly recall the terminology from
\cite{aw:curvature.three.sphere} concerning pseudo-Riemannian calculi,
as this is the context in which we shall construct a differential
calculus over the noncommutative 4-sphere.

To define a pseudo-Riemannian calculus over an algebra $\A$, we
proceed in two steps. First, we define a \emph{real metric calculus}
over an algebra $\A$ by choosing a (right) $\A$-module $M$, together with a
non-degenerate bilinear form (the metric), as well as a Lie algebra of
derivations and a map $\varphi$ that associates an element of $M$ to
each derivation. Next, a pseudo-Riemannian calculus is defined to be
a real metric calculus for which there exists a metric and
torsion-free connection on $M$.

To fix our notation and terminology, let us recall the following definitions:

\begin{definition}\label{def:metric.module}
  Let $M$ be a right $\A$-module. A map $h:M\times M\to\A$ is called
  a hermitian form on $M$ if
  \begin{align*}
    &h(U,V+W) = h(U,V) + h(U,W)\\
    &h(U,Va) = h(U,V)a\\
    &h(U,V)^\ast = h(V,U).
  \end{align*}
  A hermitian form is \emph{non-degenerate} if $h(U,V)=0$ for all
  $V\in M$ implies that $U=0$. For brevity, we simply refer to a
  non-degenerate hermitian form as a \emph{metric} on $M$. The pair $(M,h)$,
  where $M$ is a right $\A$-module and $h$ is a hermitian form on $M$,
  is called a \emph{(right) hermitian $\A$-module}. If $h$ is a
  metric, we say that $(M,h)$ is a \emph{(right) metric $\A$-module}. 
\end{definition}

\begin{definition}[\cite{aw:curvature.three.sphere}]\label{def:real.metric.calculus}
  Let $(M,h)$ be a (right) metric $\A$-module, let
  $\g\subseteq\Der(\A)$ be a (real) Lie algebra of hermitian
  derivations and let $\varphi:\g\to M$ be a $\reals$-linear
  map. If we denote the pair $(\g,\varphi)$ by $\gphi$, the
  triple $(M,h,\gphi)$ is called a \emph{real metric calculus} if
  \begin{enumerate}
  \item the image $\Mphi=\varphi(\g)$ generates $M$ as an $\A$-module,
  \item
    $h(E,E')^\ast=h(E,E')$
    for all $E,E'\in\Mphi$.
  \end{enumerate}
\end{definition}

\begin{definition}[\cite{aw:curvature.three.sphere}]
  Let $(M,h,\gphi)$ be a real metric calculus and let $\nabla$ denote
  an affine connection on $(M,\g)$. If
  \begin{align*}
    h(\nabla_dE,E') = h(\nabla_dE,E')^\ast
  \end{align*}
  for all $E,E'\in\Mphi$ and $d\in \g$ then $(M,h,\gphi,\nabla)$ is
  called a \emph{real connection calculus}.
\end{definition}

\begin{definition}[\cite{aw:curvature.three.sphere}]\label{def:pseudo.riemannian.calculus}
  Let $(M,h,\gphi,\nabla)$ be a real connection calculus. The
  calculus is \emph{metric} if
  \begin{align*}
    d\paraa{h(U,V)} = h\paraa{\nablasub{d}U,V} + h\paraa{U,\nablad V}
  \end{align*}
  for all $d\in\g$, $U,V\in M$, and \emph{torsion-free} if 
  \begin{align*}
    \nablasub{d_1}\varphi(d_2)-\nablasub{d_2}\varphi(d_1)
    -\varphi\paraa{[d_1,d_2]} = 0
  \end{align*}
  for all $d_1,d_2\in \g$. A metric and torsion-free real connection
  calculus over $M$ is called a \emph{pseudo-Riemannian calculus over $M$}. 
\end{definition}

\noindent
Given a real metric calculus $(M,h,\gphi)$, it is natural to ask if it
is possible to find an affine connection such that
$(M,h,\gphi,\nabla)$ is a pseudo-Riemannian calculus. In general, this
is not possible, but if such a connection exists, it is unique.

\begin{theorem}[\cite{aw:curvature.three.sphere}]\label{thm:levi.civita}
  Let $(M,h,\gphi)$ be a real metric calculus over $M$. Then there exists
  at most one affine connection $\nabla$ on $(M,\g)$, such that
  $(M,h,\gphi,\nabla)$ is a pseudo-Riemannian calculus. 
\end{theorem}

\subsection{Embedding of $S^4$ in $\reals^5$}
\label{sec:embedding}

\noindent
The geometric constructions for the noncommutative 4-sphere will
closely follow that of classical geometry. Therefore, let us review an
explicit parametrization of $S^4$, giving a chart that covers almost
all of the manifold. Furthermore, we present a
particular basis for vector fields over that chart.

As a subset of $\reals^5$, the 4-dimensional sphere is defined as
\begin{align*}
  S^4=\{(x^1,x^2,x^3,x^4,x^5)\in\reals^5:
  (x^1)^2+(x^2)^2+(x^3)^2+(x^4)^2+(x^5)^2=1\},
\end{align*}
and we let $U_0\subseteq S^4$ denote the chart of $S^4$ given by
\begin{alignat*}{2}
  &x^1=\cos(\xi_1)\cos(\varphi)\cos(\psi) &\qquad
  &x^2=\sin(\xi_1)\cos(\varphi)\cos(\psi)\\
  &x^3=\cos(\xi_2)\sin(\varphi)\cos(\psi) &
  &x^4=\sin(\xi_2)\sin(\varphi)\cos(\psi)\\
  &x^5=\sin(\psi),
\end{alignat*}
where $0<\xi_1,\xi_2< 2\pi$, $0< \varphi<\pi/2$ and
$-\pi/2<\psi<\pi/2$. Equivalently, one may consider $z=x^1+ix^2$,
$w=x^3+ix^4$ and $t=x^5$ with
\begin{align*}
  &z = e^{i\xi_1}\cos(\varphi)\cos(\psi)\\
  &w = e^{i\xi_2}\sin(\varphi)\cos(\psi)\\
  &t = \sin(\psi).
\end{align*}
At each point $p\in U_0$, the tangent space $T_pS^4$ is spanned by
the vectors
\begin{align*}
  &\d_{\xi_1}\xv = (-\sin\xi_1\cos\varphi\cos\psi,\cos\xi_1\cos\varphi,\cos\psi,0,0)=(-x^2,x^1,0,0,0)\\
  &\d_{\xi_2}\xv = (0,0,-\sin\xi_2\sin\varphi\cos\psi,\cos\xi_2\sin\varphi\cos\psi,0)=(0,0,-x^4,x^3,0)\\
  &\d_{\varphi}\xv = (-\cos\xi_1\sin\varphi\cos\psi,-\sin\xi_1\sin\varphi\cos\psi,\cos\xi_2\cos\varphi\cos\psi,\sin\xi_2\cos\varphi\cos\psi,0)\\
  &\d_{\psi}\xv=(-\cos\xi_1\cos\varphi\sin\psi,-\sin\xi_1\cos\varphi\sin\psi,\\
  &\qquad\qquad-\cos\xi_2\sin\varphi\sin\psi,-\sin\xi_2\sin\varphi\sin\psi,\cos\psi).
\end{align*}
These vector fields are defined in the local chart $U_0$ and we would
like to extend them to global vector fields on $S^4$ (however,
\emph{not} providing a basis at each point of $S^4$). As written
above, $\d_{\xi_1}\xv$ and $\d_{\xi_2}\xv$ may be extended to all of
$S^4$, since all components can be expressed in terms of
$x^1,\ldots,x^5$. By rescaling $\d_{\varphi}\xv$ and $\d_{\psi}\xv$
one obtains
\begin{align*}
  -&|z||w|\d_\varphi\xv = (x^1|w|^2,x^2|w|^2,-x^3|z|^2,-x^4|z|^2,0)\\
  -&\cos\psi\,\d_{\psi}\xv = (x^1t,x^2t,x^3t,x^4t,-|z|^2-|w|^2),
\end{align*}
which are well defined as vector fields on $S^4$. Thus, the globally
defined vector fields given by
\begin{alignat*}{2}
  &e_1 = (-x^2,x^1,0,0,0) &
  &e_2 = (0,0,-x^4,x^3,0)\\
  &e_3 = (x^1|w|^2,x^2|w|^2,-x^3|z|^2,-x^4|z|^2,0) &\quad
  &e_4 = (x^1t,x^2t,x^3t,x^4t,-|z|^2-|w|^2),
\end{alignat*}
span the space of vector fields over $U_0$. For later comparison, let us
write down the action of the derivations corresponding to the above
vector fields:

\begin{equation}
  \label{eq:classical.derivations}
  \begin{alignedat}{3}
    &\d_1z = iz &\quad &\d_1w = 0 &\quad &\d_1t = 0\\
    &\d_2z = 0 & &\d_w = iw & &\d_2t = 0\\
    &\d_3z = z|w|^2 & &\d_3w = -w|z|^2 & &\d_3t=0\\
    &\d_4z = zt & &\d_4w = wt & &\d_4 t = t^2-1.
  \end{alignedat}  
\end{equation}

\section{The noncommutative 4-sphere}\label{sec:nc.4.sphere}

\subsection{Basic properties of $\Sfourt$}\label{sec:basic.properties}

\noindent
For $\theta\in[0,1)$, we let $\Sfourt$ denote the unital
$\ast$-algebra (over $\complex$) generated by $Z$, $W$ and $T$,
satisfying the relations \cite{cl:isospectral,dl:instanton.algebras}
\begin{equation}\label{eq:S4.def.relations}
  \begin{split}
    &WZ=qZW\qquad \Ws Z=\qb Z\Ws\\
    &Z\Zs+W\Ws+T^2=\mid\\
    &\Ts=T\qquad [T,Z]=[T,W]=[W,\Ws]=[Z,\Zs]=0,
  \end{split}
\end{equation}
where $q=e^{i2\pi\theta}$. Furthermore, $Z\Zs\in\ZSfourt$ and
  $W\Ws\in\ZSfourt$ where $\ZSfourt$ denotes the center of
  $\Sfourt$. It follows from \eqref{eq:S4.def.relations} that a
linear basis for $\Sfourt$ is given by the elements
\begin{align*}
  Z^j(\Zs)^kW^l(\Ws)^mT^\epsilon
\end{align*}
for $j,k,l,m\in\{0,1,2,\ldots\}$ and $\epsilon\in\{0,1\}$ (where, e.g., higher
powers of $T$ are eliminated by using the relation $T^2=\mid-Z\Zs-W\Ws$). For
convenience, let us introduce the multi-index notation
$I=(j,k,l,m,\epsilon)$ and
\begin{align*}
  e^I = Z^j(\Zs)^kW^l(\Ws)^mT^\epsilon
\end{align*}
such that, in this notation, every element $a\in\Sfourt$ can uniquely
be written as
\begin{align*}
  a =\sum_{I}a_Ie^I
\end{align*}
with $a_I\in\complex$. It is useful to develop the multi-index
notation a bit further. Namely, for $I=(j,k,l,m,\epsilon)$ we write
$I=(\Ih,\epsilon)$ with $\Ih=(j,k,l,m)$. Furthermore, we introduce
\begin{align*}
  1_Z = (1,1,0,0,0) = (\oneh_Z,0)\quad\text{and}\quad
  1_W = (0,0,1,1,0) = (\oneh_W,0),
\end{align*}
and we write $I+J$ for component-wise addition of multi-indices.  Let
us now state the result of multiplying two basis elements in the following
lemma:
\begin{lemma}\label{lemma:basis.product}
  If $I_1=(j_1,k_1,l_1,m_1,\epsilon_1)$ and 
  $I_2=(j_2,k_2,l_2,m_2,\epsilon_2)$ then
  \begin{align*}
    e^{I_1}e^{I_2} = 
    \begin{cases}
      q^{(l_1-m_1)(j_2-k_2)}e^{I_1+I_2}\text{ if }
        \epsilon_1+\epsilon_2\leq 1\\
      q^{(l_1-m_1)(j_2-k_2)}
      \paraa{e^{(\Ih_1+\Ih_2,0)}-e^{(\Ih_1+\Ih_2+\oneh_Z,0)}-e^{(\Ih_1+\Ih_2+\oneh_W,0)}}
      \text{ if }
      \epsilon_1+\epsilon_2=2.
    \end{cases}
  \end{align*}
\end{lemma}

\begin{proof}
  Using \eqref{eq:S4.def.relations} one obtains
  \begin{align*}
    e^Ie^J &= Z^{j_1}(\Zs)^{k_1}W^{l_1}(\Ws)^{m_1}T^{\epsilon_1}
    Z^{j_2}(\Zs)^{k_2}W^{l_2}(\Ws)^{m_2}T^{\epsilon_2}\\
    &=q^{j_2(l_1-m_1)}Z^{j_1+j_2}(\Zs)^{k_1}W^{l_1}(\Ws)^{m_1}
    (\Zs)^{k_2}W^{l_2}(\Ws)^{m_2}T^{\epsilon_1+\epsilon_2}\\
    &=q^{j_2(l_1-m_1)}q^{k_2(m_1-l_1)}Z^{j_1+j_2}(\Zs)^{k_1+k_2}W^{l_1}(\Ws)^{m_1}
    W^{l_2}(\Ws)^{m_2}T^{\epsilon_1+\epsilon_2}\\
    &=q^{(l_1-m_1)(j_2-k_2)}
    Z^{j_1+j_2}(\Zs)^{k_1+k_2}W^{l_1+l_2}(\Ws)^{m_1+m_2}T^{\epsilon_1+\epsilon_2}.
  \end{align*}
  Now, if $\epsilon_1+\epsilon_2\leq 1$ then the statement in the
  Lemma is proved. If $\epsilon_1+\epsilon_2=2$, then the statement
  follows after using that $T^2 = \mid-Z\Zs-W\Ws$, and the fact that
  both $Z\Zs$ and $W\Ws$ are central.
\end{proof}

\noindent
Let us now proceed to state a few properties of $\Sfourt$ that we
shall need in the following.

\begin{proposition}\label{prop:zero.divisors}
  The elements $Z\Zs$, $W\Ws$ and $\mid-T^2$ are regular (i.e. none of
  them is a zero divisor).
\end{proposition}

\begin{proof}
  Let us first prove that $Z\Zs$ is not a zero divisor. Thus,
  let $a$ be an element of $\Sfourt$, given as
  \begin{align*}
    a=\sum_I a_Ie^I
  \end{align*}
  and compute (by using Lemma~\ref{lemma:basis.product})
  \begin{align*}
    &Z\Zs a=\sum_I a_Ie^{1_Z}e^I
    =\sum_Iq^{(0-0)(j-k)}a_Ie^{I+1_Z}
    =\sum_Ia_Ie^{I+1_Z}.
  \end{align*}
  Clearly, setting $Z\Zs a = 0$ gives $a_I=0$ for all $I$ since
    $\{e^I\}$ is a basis for $\Sfourt$.  Similarly,
  we consider
  \begin{align*}
    W\Ws a = \sum_Ia_Ie^{1_W}e^I
    =\sum_Iq^{(1-1)(j-k)}e^{I+1_W}=
    \sum_Ia_Ie^{I+1_W}
  \end{align*}
  and conclude that $W\Ws a=0$ gives $a=0$. Finally, we compute
  \begin{align*}
    (\mid-T^2)a &=(\Zsq+\Wsq)a
    =\sum_{I}a_I\paraa{e^{1_Z}+e^{1_W}}e^I
    =\sum_Ia_Ie^{I+1_Z}+\sum_Ia_Ie^{I+1_W}\\
    &=\sum_{j=0,\,l,m\geq 1}\!\!\!\!a_{I-1_W}e^I
    +\!\!\!\sum_{k=0,\,j,l,m\geq 1}\!\!\!\!a_{I-1_W}e^I
    +\sum_{l=0,\,j,k\geq 1}\!\!\!\!a_{I-1_Z}e^I\\
    &\qquad+\sum_{m=0,\,j,k,l\geq 1}\!\!\!\!a_{I-1_Z}e^I
    +\sum_{j,k,l,m\geq 1}\paraa{a_{I-1_Z}+a_{I-1_W}}e^I.
  \end{align*}
  Note that in the above expression, every basis element appears at
  most once.  Therefore, setting $(\mid-T^2)a=0$ immediately gives
  $a_{j,k,l,m,\epsilon}=0$ if at least one of $j,k,l,m$ is zero. If
  $j,k,l,m\geq 1$ one gets
  \begin{align*}
    a_{I-(0,0,1,1,0)}=-a_{I-(1,1,0,0,0)}\implies a_{I}=-a_{I+(1,1,-1,-1)},
  \end{align*}
  which, by iteration, gives
  \begin{align*}
    a_{I} = (-1)^na_{I+(n,n,-n,-n)} \quad\text{for }0\leq n\leq\min(l,m).
  \end{align*}
  Hence, since $a_{j,k,l,m,\epsilon}=0$ if at least one of $j,k,l,m$
  is zero, one concludes that
  \begin{align*}
    a_{(j,k,l,m,\epsilon)} = 
    \begin{cases}
      (-1)^la_{j+l,k+l,0,m-l,\epsilon}=0\text{ if }l\leq m\\
      (-1)^ma_{j+m,k+m,l-m,0,\epsilon}=0\text{ if }l\geq m\\
    \end{cases}
  \end{align*}
  which, together with the previous observation, shows that $a=0$.
\end{proof}

\noindent
We have already noted that $Z\Zs$, $W\Ws$ and $T$ are central
elements. The next results shows that if $\theta$ is an irrational
number, then these elements generate the center of $\Sfourt$.

\begin{proposition}\label{prop:algebra.center}
  If $\theta$ is irrational then $\ZSfourt$ is generated by $Z\Zs$,
  $W\Ws$ and $T$. That is, every $a\in\ZSfourt$ can be uniquely
  written as
  \begin{align*}
    a = \sum_{j,k,\epsilon}a_{jk\epsilon}(Z\Zs)^j(W\Ws)^kT^{\epsilon}
  \end{align*}
  where $a_{jk\epsilon}\in\complex$, $j,k\in\{0,1,2,\ldots\}$ and $\epsilon\in\{0,1\}$.
\end{proposition}

\begin{proof}
  Let $a$ be an arbitrary (nonzero) central element of $\Sfourt$ and write
  \begin{align*}
    a = \sum_I a_Ie^I.
  \end{align*}
  In particular, $a$ has to commute with $Z$, and one computes
  \begin{align*}
    [a,Z] = \sum_Ia_I\parab{e^Ie^{(1,0,0,0,0)}-e^{(1,0,0,0,0)}e^I}
    =\sum_Ia_I(q^{l-m}-1)e^{I+(1,0,0,0,0)}.
  \end{align*}
  Demanding that $[a,Z]=0$ gives $(q^{l-m}-1)a_I=0$. If $a\neq 0$,
  there exists an $I$ such that $a_I\neq 0$, which implies that
  $q^{l-m}=1$. Since $\theta$ is assumed to be irrational it follows
  that $l=m$. Similarly, if $a$ commutes with $W$ then 
  \begin{align*}
    0 = [a,W] = \sum_{I}a_I\parab{e^Ie^{(0,0,1,0,0)}-e^{(0,0,1,0,0)}e^I}
    =\sum_Ia_I\paraa{1-q^{j-k}}e^{I+(0,0,1,0,0)}
  \end{align*}
  giving $j=k$ in analogy with the previous case. Thus, an element
  $a\in\ZSfourt$ must be of the following form
  \begin{align*}
    a = \sum_{j,k,\epsilon}a_{j,k,\epsilon}\paraa{Z\Zs}^j\paraa{W\Ws}^kT^\epsilon,
  \end{align*}
  and it is clear that any element of the above form is in $\ZSfourt$
  since $Z\Zs$, $W\Ws$ and $T$ are central.
\end{proof}

\begin{remark}
  Note that Proposition~\ref{prop:algebra.center} does not hold if
  $\theta$ is rational. For instance, if $q^N=1$ then both $Z^N$ and
  $W^N$ are central elements.
\end{remark}

\noindent
Let us introduce
\begin{alignat*}{2}
  &X^1 = \half\paraa{Z + Z^\ast} &\qquad
  &X^2 = \frac{1}{2i}\paraa{Z-Z^\ast}\\
  &X^3 = \half\paraa{W + W^\ast} &\qquad
  &X^4 = \frac{1}{2i}\paraa{W-W^\ast}\\
  &\Zsq = ZZ^\ast\quad\Wsq = WW^\ast & & X^5=T,
\end{alignat*}
and note that $\Zsq = (X^1)^2 + (X^2)^2$ and 
$\Wsq=(X^3)^2 + (X^4)^2$, as well as
\begin{align*}
  (X^1)^2 + (X^2)^2 + (X^3)^2 + (X^4)^2 +(X^5)^2 = \Zsq+\Wsq+T^2 = \mid.
\end{align*}
Moreover, the normality of $Z$ and $W$ is equivalent to
$[X^1,X^2]=[X^3,X^4]=0$. Next, let us show that there exist
noncommutative analogues of the four derivations appearing in
\eqref{eq:classical.derivations}.

\begin{proposition}\label{prop:derivations}
  There exist hermitian derivations $\dt_1$, $\dt_2$,
  $\dt_3$, $\dt_4$ such that
  \begin{alignat*}{3}
    &\dt_1 Z = iZ &\quad &\dt_1W= 0 & &\dt_1T=0\\
    &\dt_2 Z = 0 & &\dt_2W=iW &\quad &\dt_2T=0\\
    &\dt_3 Z = Z\Wsq & &\dt_3W =-W\Zsq &\quad &\dt_3T=0\\
    &\dt_4 Z = ZT & &\dt_4W = WT & &\dt_4T = T^2-\mid,
  \end{alignat*}
  and it follows that
  \begin{align*}
    &[\dt_1,\dt_2]=[\dt_1,\dt_3]=[\dt_1,\dt_4]=0\\
    &[\dt_2,\dt_3]=[\dt_2,\dt_4]=0\\
    &[\dt_3,\dt_4]=-2T\dt_3.
  \end{align*}
\end{proposition}

\begin{proof}
  If the derivations exist, the relations given above
  (together with the fact that they are hermitian derivations),
  completely determine their actions via Leibniz' rule. However, for
  these derivations to be well-defined, one has to check that they
  respect the defining relations \eqref{eq:S4.def.relations} of $\Sfourt$.
  For instance
  \begin{align*}
    \dt_1(WZ-qZW) &= (\dt_1W)Z+W(\dt_1Z)-q(\dt_1Z)W-qZ(\dt_1W)\\
    &=iWZ-iqZW = i(WZ-qZW) = 0,
  \end{align*}
  and
  \begin{align*}
    \dt_3(WZ-qZW) &= (\dt_3W)Z+W(\dt_3Z)-q(\dt_3Z)W-qZ(\dt_3W)\\
    &=-W\Zsq Z + WZ\Wsq - qZ\Wsq W+qZW\Zsq\\
    &=(WZ-qZW)\Wsq - (WZ-qZW)\Zsq=0
  \end{align*}
  (using that $\Zsq$ and $\Wsq$ are central).  In this way, relations
  \eqref{eq:S4.def.relations} can be checked for the derivations
  $\dt_1,\dt_2,\dt_3,\dt_4$.
\end{proof}

\subsection{A real metric calculus over $\Sfourt$}\label{sec:real.metric.calculus}

\noindent
In this section, we shall introduce a differential calculus over
$\Sfourt$ in close analogy with the classical parametrization in
Section \ref{sec:embedding}. The calculus will be constructed in the
framework of pseudo-Riemannian calculi, as developed in
\cite{aw:curvature.three.sphere}, and briefly reviewed in
Section~\ref{sec:pseudo.riemannian.calculi}.

To this end, we introduce four elements of the free (right) module
$(\Sfourt)^5$ that correspond to the classical vector fields
$e_1,e_2,e_3,e_4$ in Section~\ref{sec:embedding}. However, in order to
properly define a connection, one needs to slightly rescale $e_1$ and
$e_2$. Thus, we consider the following elements of $(\Sfourt)^5$:
\begin{align*}
  E_1 &= (-X^2(\mid-T^2),X^1(\mid-T^2),0,0,0)\\ 
  E_2 &= (0,0,-X^4(\mid-T^2),X^3(\mid-T^2),0)\\
  E_3 &= (X^1\Wsq,X^2\Wsq,-X^3\Zsq,-X^4\Zsq,0)\\
  E_4 &= (X^1T,X^2T,X^3T,X^4T,T^2-\mid),
\end{align*}
and let $M$ be the submodule of $(\Sfourt)^5$ generated by
$\{E_1,E_2,E_3,E_4\}$.  Note that there are no ordering ambiguities when
defining these elements, since $\Zsq$, $\Wsq$ and $T$ are
central. This module is the analogue of the local vector fields over
the chart $U_0$, and the corresponding local triviality is reflected
in the following result.

\begin{proposition}
  The module $M=\{E_1a+E_2b+E_3c+E_4d:a,b,c,d\in\Sfourt\}$
  is a free (right) $\Sfourt$-module of rank $4$, and $\{E_1,E_2,E_3,E_4\}$ is a basis for $M$.
\end{proposition}

\begin{proof}
  By definition, $\{E_1,E_2,E_3,E_4\}$ generates $M$. To prove that
  $\{E_1,E_2,E_3,E_4\}$ is a basis, we assume that
  \begin{align}\label{eq:free.module.assumption}
    E_1a+E_2b+E_3c+E_4d=0
  \end{align}
  and show that this implies that $a=b=c=d=0$. Relation
  (\ref{eq:free.module.assumption}) is equivalent to the equations
  \begin{align*}
    -&X^2(\mid-T^2)a+X^1\Wsq c+X^1Td=0\\
    &X^1(\mid-T^2)a+X^2\Wsq c+X^2Td=0\\
    -&X^4(\mid-T^2)b-X^3\Zsq c+X^3Td=0\\
    &X^3(\mid-T^2)b-X^4\Zsq c + X^4Td=0\\
    &(\mid-T^2)d=0,
  \end{align*}
  which immediately implies that $d=0$ (since $\mid-T^2$ is not a zero
  divisor by Proposition~\ref{prop:zero.divisors}), and the remaining
  equations may be written as
  \begin{align}
    -&X^2(\mid-T^2)a+X^1\Wsq c=0\label{eq:fm.first}\\
    &X^1(\mid-T^2)a+X^2\Wsq c=0\label{eq:fm.second}\\
    -&X^4(\mid-T^2)b-X^3\Zsq c=0\\
    &X^3(\mid-T^2)b-X^4\Zsq c=0.
  \end{align}
  The sum of \eqref{eq:fm.first}, multiplied from the left with $X^1$,
  and \eqref{eq:fm.second}, multiplied from the left by $X^2$ gives
  \begin{align*}
    \paraa{(X^1)^2+(X^2)^2}\Wsq c = \Zsq\Wsq c = 0,
  \end{align*}
  (using that $[X^1,X^2]=0$) which implies that $c=0$ since neither
  $\Zsq$ nor $\Wsq$ is a zero divisor (by
  Proposition~\ref{prop:zero.divisors}). Hence, one is left with the
  equations
  \begin{alignat*}{2}
    &X^2(\mid-T^2)a = 0 &\qquad &X^1(\mid-T^2)a=0\\
    &X^4(\mid-T^2)b=0 & &X^3(\mid-T^2)b = 0,
  \end{alignat*}
  and since $\mid-T^2$ is not a zero divisor one obtains
  \begin{alignat*}{2}
    &X^2a = 0 &\qquad &X^1a=0\\
    &X^4b=0 & &X^3b = 0,
  \end{alignat*}
  giving 
  \begin{align*}
    &\paraa{(X^1)^2+(X^2)^2}a=\Zsq a=0\\
    &\paraa{(X^3)^2+(X^4)^2}b=\Wsq b=0,
  \end{align*}
  which implies that $a=b=0$. Thus, we have shown that
  $E_1a+E_2b+E_3c+E_4d=0$ necessarily gives $a=b=c=d=0$, which proves
  that $\{E_1,E_2,E_3,E_4\}$ is indeed a basis for $M$.
\end{proof}

\noindent 
In the module $M$, we introduce the restriction of the canonical
metric on $(\Sfourt)^5$:
\begin{align*}
  h(U,V) = \sum_{a,b=1}^4(U^a)^\ast h_{ab}V^b
\end{align*}
for $U=E_aU^a$ and $V=E_bV^b$, where
\begin{align*}
  h_{ab}=\sum_{i=1}^5(E_a^i)^\ast (E_b^i),
\end{align*}
giving
\begin{align*}
  (h_{ab})=
  \begin{pmatrix}
    \Zsq(\mid-T^2)^2 & 0 & 0 & 0\\
    0 & \Wsq(\mid-T^2)^2 & 0 & 0\\
    0 & 0 & \Zsq\Wsq(\mid-T^2) & 0\\
    0 & 0 & 0 & \mid-T^2
  \end{pmatrix}.
\end{align*}
As we shall be interested in perturbations of the standard metric, we
introduce
\begin{align*}
  \hd = \delta h
\end{align*}
where $\delta\in\Sfourt$ is assumed to be a hermitian, central and
regular element.  Since $\hd$ is diagonal,
and each diagonal element is regular, it follows immediately that
$\hd$ is non-degenerate on $M$; i.e.
\begin{align*}
  h(U,V)=0\text{ for all }V\in M\implies U=0.
\end{align*}
Thus, the pair $(M,\hd)$ is a metric module (cf. Definition~\ref{def:metric.module}). To construct a real metric calculus
over $(M,\hd)$ (cf. Definition~\ref{def:real.metric.calculus}),
we need to associate derivations to $E_1,E_2,E_3,E_4$. In analogy with
the classical situation, we consider the following derivations
\begin{alignat*}{2}
  &\d_1 = (\mid-T^2)\dt_1 &\quad &\d_2=(\mid-T^2)\dt_2\\
  &\d_3 = \dt_3 & &\d_4=\dt_4,
\end{alignat*}
with $\dt_1,\dt_2,\dt_3,\dt_4$ given as in
Proposition~\ref{prop:derivations}. (Note that $\d_1$ and $\d_2$ are
derivations since $\mid-T^2$ is central.) These derivations generate
an infinite-dimensional Lie algebra.
\begin{proposition}
  For $n\in\naturals_0$, the hermitian derivations
  \begin{align*}
    \dn_1=T^n(\mid-T^2)\dt_1,\quad
    \dn_2=T^n(\mid-T^2)\dt_2,\quad
    \dn_3=T^n\dt_3,\quad
    \d_4=\dt_4
  \end{align*}
  span an infinite-dimensional Lie algebra, where
  \begin{align*}
    &[\dn_1,\dn_2]=[\dn_1,\dn_3]=[\dn_2,\dn_3]=0\\
    &[\d_4,\dn_i] = (n+2)\dn[n+1]_i-n\dn[n-1]_i,
  \end{align*}
  for $i=1,2,3$ (with the convention that $n\dn[n-1]_i= 0$ if
  $n=0$). Moreover, it follows that
  \begin{alignat*}{3}
    &\d_1\Zsq = 0 & &\d_1\Wsq = 0 & &\d_1(\mid-T^2)=0\\
    &\d_2\Zsq = 0 & &\d_2\Wsq = 0 & &\d_2(\mid-T^2)=0\\
    &\d_3\Zsq = 2\Zsq\Wsq &\quad &\d_3\Wsq=-2\Zsq\Wsq &\quad &\d_3(\mid-T^2)=0\\
    &\d_4\Zsq = 2\Zsq T & &\d_4\Wsq = 2\Wsq T &
    &\d_4(\mid-T^2) = 2T(\mid-T^2),
  \end{alignat*}
  where $\d_i\equiv \dn[0]_i$ for $i=1,2,3$.
\end{proposition}

\begin{proof}
  The proof consists of straight-forward computations using the
  definition of $\dt_1,\dt_2,\dt_3,\dt_4$ in Proposition~\ref{prop:derivations}.
\end{proof}

\noindent
We let $\g$ denote the (real) Lie algebra spanned by
$\dn_1,\dn_2,\dn_3,\d_4$, and let $\varphi:\g\to M$ be the
$\reals$-linear map defined by
\begin{align*}
  &\varphi(\dn_i)= E_iT^n\quad\text{for }i=1,2,3,\\
  &\varphi(\d_4)=E_4.
\end{align*}
The pair $(\g,\varphi)$ is denoted by $\gphi$.

\begin{proposition}
  The triple $(M,\hd,\gphi)$ is a real metric calculus over $\Sfourt$.
\end{proposition}

\begin{proof}
  As already noted, the metric $\hd$ is non-degenerate on $M$ and, by
  definition, $\{E_1,E_2,E_3,E_4\}$ generates $M$, which implies that
  the image of $\varphi$ generates $M$. Finally, since every component
  of $\hd$ is hermitian, it follows that $\hd(E,E')$ is hermitian for
  all $E,E'$ in the image of $\varphi$. This shows that the triple
  $(M,\hd,\gphi)$ satisfies all the requirements of a real metric
  calculus.
\end{proof}

\noindent
Given a real metric calculus $(M,\hd,\gphi)$, there exists at most one
metric and torsion-free connection on the module $M$
(cf. Theorem~\ref{thm:levi.civita}). In
Section~\ref{sec:pseudo.Riemannian} we proceed to show that such a
connection exists, but let us first discuss certain aspects of
localization on $\Sfourt$.

\subsection{The local algebra $\Sfourtloc$}\label{sec:local.algebra}

For the classical $4$-sphere, the vector fields corresponding to
$E_1,E_2,E_3,E_4$ are linearly independent in the chart given in
Section~\ref{sec:embedding}. Thus, as already mentioned, the module
$M$ does not correspond to the module of vector fields of $S^4$, but
rather to a local trivialization in the chart $U_0$. In this chart,
the functions $|w|^2$, $|z|^2$ and $1-t^2$ are invertible, and in
analogy with this situation we shall introduce a localization of the
algebra $\Sfourt$ in order to be able to perform computations in a
``noncommutative chart''. Moreover, let us also consider the inverse
of $\mid+T^2$ (which is globally invertible in the classical setting)
as it is an algebraic prototype of the kind of perturbations of the
metric that we will consider.  To this end, we let $S$ be the
multiplicative subset of $\Sfourt$ generated by $\mid$, $\Zsq$,
$\Wsq$, $\mid-T^2$ and $\mid+T^2$. Since every element of $S$ is
central, $S$ trivially fulfills the left (and right) Ore condition
\cite{o:lineareqnc}. Hence, the localization of $\Sfourt$ at $S$
exists, and we denote it by $\Sfourtloc$. In other words, $\Sfourtloc$
is constructed from $\Sfourt$ by adding the formal inverses of $\Zsq$,
$\Wsq$, $\mid-T^2$ and $\mid+T^2$. Clearly, $(M,\hd,\gphi)$, as
constructed above, is also a real metric calculus over
$\Sfourtloc$. In what follows, we shall discuss the two algebras in
parallel.

Let us take a closer look at the structure of the noncommutative
localization we have introduced. The algebra $\Sfourt$ has been
localized to include elements, which are classically not globally
defined, and the corresponding free module $M$ has been defined, which we
claim to be the local trivialization of the module of vector fields. Now,
is there a global module of vector fields, for which $M$ is a
localization? For the noncommutative 4-sphere, a particular projective
module presents itself as a natural candidate.  Defining
$\P:(\Sfourt)^5\to(\Sfourt)^5$ as
\begin{align}\label{eq:P.def}
  \P(U)=\sum_{j=1}^5\paraa{\delta^{ij}\mid-X^iX^j}U^j
\end{align}
where $U=e_iU^i$, it is easy to check that $\P^2=\P$ since
\begin{align*}
  (X^1)^2+(X^2)^2+(X^3)^2+(X^4)^2+(X^5)^2 = \mid.
\end{align*}
Let us denote the image of $\P$ by $\TSfourt$, which is, by
definition, a finitely generated projective module.  In classical
geometry, $\P$ is the projector that defines the module of vector
fields on $S^4$. Let us now show that, over the local algebra
$\Sfourtloc$, this module is isomorphic to the module of the real
metric calculus we have previously constructed.

\begin{proposition}\label{prop:local.global.isomorphic}
  The modules $\TSfourt$ and $M$ are isomorphic as
  right $\Sfourtloc$-modules.
\end{proposition}

\begin{proof}
  First of all, it is easy to check that $E_1,E_2,E_3,E_4\in\TSfourt$;
  for instance,
  \begin{align*}
    \sum_{i=1}^5X^iE_1^i = X^1(-X^2)+X^2X^1 = 0,
  \end{align*}
  since $[X^1,X^2] = 0$, which implies that $\P(E_1)=E_1$ and
  $E_1\in\TSfourt$. Thus, it follows that $M\subseteq \TSfourt$. Next,
  we will show that $\TSfourt\subseteq M$, by explicitly writing
  $\P(e_i)$ (for $i=1,2,3,4,5$) as linear combinations of
  $E_1,E_2,E_3,E_4$. Since $\{\P(e_i)\}_{i=1}^5$ generates $\TSfourt$,
  this shows that every element of $\TSfourt$ can be written in terms
  of $E_1,E_2,E_3,E_4$. We claim that
  \begin{align*}
    \P(e_1) &= 
              -E_1X^2\Zsqi\tfi
              +E_3X^1\Zsqi\tfi
              +E_4X^1T\tfi\\
    \P(e_2) &= 
              E_1X^1\Zsqi\tfi
              +E_3X^2\Zsqi\tfi
              +E_4X^2T\tfi\\
    \P(e_3) &=
              -E_2X^4\Wsqi\tfi
              -E_3X^3\Wsqi\tfi
              +E_4X^3T\tfi\\
    \P(e_4) &=
              E_2X^3\Wsqi\tfi
              -E_3X^4\Wsqi\tfi
              +E_4X^4T\tfi\\
    \P(e_5) &= -E_4.
  \end{align*}
  Let us show that $\P(e_1)$ can be written as the linear combination
  given above. The proof of the other four identities is
  analogous. First, one checks that
  \begin{align*}
    \P(e_1) = \paraa{1-(X^1)^2,-X^2X^1,-X^3X^1,-X^4X^1,-X^5X^1}.
  \end{align*}
  Next, write
  \begin{align*}
    U&=-E_1X^2\Zsqi\tfi+E_3X^1\Zsqi\tfi+E_4X^1T\tfi\\
     &=(U^1,U^2,U^3,U^4,U^5),
  \end{align*}
  and compute the components one by one
  \begin{align*}
    U^1 &= (X^2)^2\Zsqi+(X^1)^2\Wsq\Zsqi\tfi+(X^1)^2T^2\tfi\\
    &=(X^2)^2\Zsqi+(X^1)^2\Wsq\Zsqi\tfi\\
    &\qquad-(X^1)^2(\mid-T^2)\tfi+(X^1)^2\tfi\\
    &=-(X^1)^2+\Zsqi\tfi\parab{(X^2)^2(\mid-T^2)+(X^1)^2(\Zsq+\Wsq)}\\
    &=\Big/ \text{using }\Zsq+\Wsq+T^2=\mid \Big /\\
    &=-(X^1)^2+\Zsqi\tfi\parab{(X^2)^2(\mid-T^2)+(X^1)^2(\mid-T^2)}\\
    &=-(X^1)^2+\Zsqi\paraa{(X^1)^2+(X^2)^2}=\mid-(X^1)^2,\\
    U^2 &= -X^1X^2\Zsqi + X^2X^1\Wsq\Zsqi\tfi+X^2X^1T^2\tfi\\
    &= \Big/ \text{using }[X^1,X^2]=0 \Big /\\
    &= -X^2X^1\Zsqi\tfi\paraa{\mid-T^2-\Wsq}+X^2X^1T^2\tfi\\
    &=-X^2X^1\Zsqi\tfi\Zsq+X^2X^1T^2\tfi\\
    &=-X^2X^1\tfi\paraa{\mid-T^2} = -X^2X^1,\\
    U^3 &= -X^3X^1\tfi+X^3X^1T^2\tfi\\
    &=-X^3X^1\tfi(\mid-T^2) = -X^3X^1,\\
    U^4 &= -X^4X^1\tfi+X^4X^1T^2\tfi\\
    &=-X^4X^1\tfi(\mid-T^2) = -X^4X^1,\\
    U^5 &= (T^2-\mid)X^1T\tfi=-X^1T = -X^1X^5.
  \end{align*}
  Thus, we have shown that 
  \begin{align*}
    \P(e_1) &= -E_1X^2\Zsqi\tfi+E_3X^1\Zsqi\tfi+E_4X^1T\tfi,
  \end{align*}
  which, together with the other four analogous computations, shows
  that $\TSfourt$ is contained in $M$. Combined with the fact that
  $M\subseteq \TSfourt$ one can conclude that $\TSfourt=M$ as right
  $\Sfourtloc$-modules.
\end{proof}

\subsection{Pseudo-Riemannian calculus}\label{sec:pseudo.Riemannian}

To construct a connection $\nabla$ on $M$, such that
$(M,\hd,\gphi,\nabla)$ is a pseudo-Riemannian calculus, we consider
the following class of perturbations.  Let us assume that
\begin{align*}
  \d_a\delta = 2\alpha_a\delta,
\end{align*}
where $\alpha_a\in\Sfourtloc$ is hermitian, for $a=1,2,3,4$. The
connection will be constructed over $\Sfourtloc$, but we shall see that
perturbations in certain directions give connections over
$\Sfourt$. 

\begin{proposition}\label{prop:4sphere.connection}
  Let $\delta\in\Sfourtloc$ be a hermitian, regular and central element, such
  that $\d_a\delta=2\alpha_a\delta$,
  for $a=1,2,3,4$, where $\alpha_a\in\Sfourtloc$ and
  $\alpha_a^\ast=\alpha_a$. Then there exists a unique connection
  $\nabla$, such that $(M,\hd,\gphi,\nabla)$ is a pseudo-Riemannian
  calculus over $\Sfourtloc$, and $\nabla$ is given by
  \begin{align*}
    \nabla_1E_1 &= 
    E_1\alpha_1 - E_2\alpha_2\Zsq\Wsqi
    -E_3\paraa{\alpha_3\Wsqi+\mid}(\mid-T^2)\\
    &\qquad-E_4(\alpha_4+T)\Zsq(\mid-T^2)\\
    \nabla_1E_2 &= \nabla_2E_1 = E_1\alpha_2+E_2\alpha_1\\
    \nabla_1E_3 &= \nabla_3E_1 = E_1(\alpha_3+\Wsq) + E_3\alpha_1\\
    \nabla_1E_4 &= E_1(\alpha_4+T)+E_4\alpha_1\\
    \nabla_4E_1 &= E_1(\alpha_4+3T)+E_4\alpha_1\\
    \nabla_2E_2 &= 
      -E_1\alpha_1\Wsq\Zsqi + E_2\alpha_2
      -E_3\paraa{\alpha_3\Zsqi-\mid}(\mid-T^2)\\
    &\qquad-E_4(\alpha_4+T)\Wsq(\mid-T^2)\\
    \nabla_2E_3&=\nabla_3E_2=E_2(\alpha_3-\Zsq)+E_3\alpha_2\\
    \nabla_2E_4 &= E_2(\alpha_4+T)+E_4\alpha_2\\
    \nabla_4E_2 &= E_2(\alpha_4+3T)+E_4\alpha_2\\
    \nabla_3E_3 &= -E_1\alpha_1\Wsq(\mid-T^2)^{-1}
                  -E_2\alpha_2\Zsq(\mid-T^2)^{-1}\\
                &\qquad+E_3(\alpha_3+\Wsq-\Zsq)-E_4(\alpha_4+T)\Zsq\Wsq\\
    \nabla_3E_4 &= E_3(\alpha_4+T)+E_4\alpha_3\\
    \nabla_4E_3 &= E_3(\alpha_4+3T)+E_4\alpha_3\\
    \nabla_4E_4 &= 
                  -E_1\alpha_1\Zsqi(\mid-T^2)^{-1}
                  -E_2\alpha_2\Wsqi(\mid-T^2)^{-1}\\
                &\qquad-E_3\alpha_3\Zsqi\Wsqi+E_4(\alpha_4+T),
  \end{align*}
  and
  \begin{align*}
    \nabla_{\dn_i}E_a=\paraa{\nabla_iE_a}T^n
  \end{align*}
  for $i=1,2,3$, $a=1,2,3,4$, where $\nabla_a\equiv\nabla_{\d_a}$.
\end{proposition}

\begin{proof}
  Let us recall (cf. \cite{aw:curvature.three.sphere}) that Kozul's formula
  \begin{equation}\label{eq:kozul.formula}
    \begin{split}
      2h&(\nablasub{d_1}E_2,E_3)=
      d_1h(E_2,E_3)+d_2h(E_3,E_1)-d_3h(E_1,E_2)\\
      &\quad
      -h\paraa{E_1,\varphi([d_2,d_3])}
      +h\paraa{E_2,\varphi([d_3,d_1])}
      +h\paraa{E_3,\varphi([d_1,d_2])},
    \end{split}
  \end{equation}
  where $E_1,E_2,E_3\in\Mphi$ and $d_1,d_2,d_3\in\g$, gives a
  straight-forward way of finding a connection on $M$ such that
  $(M,\hd,\gphi,\nabla)$ is a pseudo-Riemannian calculus. Namely, if
  one finds $U_{ab}\in M$ such that
  \begin{equation}\label{eq:Kozul.Uab}
    \begin{split}
      2h&(U_{ab},E_c)=
      \d_ah(E_b,E_c)+\d_bh(E_a,E_c)-\d_ch(E_a,E_b)\\
      &\quad
      -h\paraa{E_a,\varphi([\d_b,\d_c])}
      +h\paraa{E_b,\varphi([\d_c,\d_a])}
      +h\paraa{E_c,\varphi([\d_a,\d_b])}
    \end{split}
  \end{equation}
  for all $a,b,c\in\{1,2,3,4\}$ then (since the module $M$ is free)
  one may set $\nabla_aE_b=U_{ab}$, and it follows that
  $(M,\hd,\gphi,\nabla)$ is a pseudo-Riemannian calculus (see
  Corollary~3.8 in \cite{aw:curvature.three.sphere}). It is
  straight-forward to check that the expressions given in
  Proposition~\ref{prop:4sphere.connection} fulfill
  \eqref{eq:Kozul.Uab}. For instance, to check Kozul's formula for
  $\nabla_1E_1$ one sets
  \begin{align*}
    K_a = 
    \hd(\nabla_1E_1,E_a)-\d_1\hd(E_1,E_a)+\half\d_a\hd(E_1,E_1)
    +\hd\paraa{E_1,\varphi([\d_1,\d_a])}
  \end{align*}
  which gives
  \begin{align*}
    K_1 &= \hd(\nabla_1E_1,E_1) -\alpha_1\delta\Zsq(\mid-T^2)^2\\
    &= \alpha_1\hd(E_1,E_1)-\alpha_1\delta\Zsq(\mid-T^2)^2 \\
    &= \alpha_1\delta\Zsq(\mid-T^2)^2-\alpha_1\delta\Zsq(\mid-T^2)^2 =0,\\
    K_2 &= \hd(\nabla_1E_1,E_2) +\alpha_2\delta\Zsq(\mid-T^2)^2\\
    &=-\alpha_2\Zsq\Wsqi\hd(E_2,E_2)+\alpha_2\delta\Zsq(\mid-T^2)^2\\
    &=-\alpha_2\Zsq\Wsqi\Wsq(\mid-T^2)^2+\alpha_2\delta\Zsq(\mid-T^2)^2=0,\\
    K_3 &=\hd(\nabla_1E_1,E_3)+\thalf\d_3\paraa{\delta\Zsq(\mid-T^2)^2}\\
    &=-(\mid+\alpha_3\Wsqi)(\mid-T^2)\hd(E_3,E_3)+\paraa{\alpha_3\delta\Zsq
    +\delta\Zsq\Wsq}(\mid-T^2)^2\\
    &=-(\mid+\alpha_3\Wsqi)\delta\Zsq\Wsq(\mid-T^2)^2+\paraa{\alpha_3\delta\Zsq
    +\delta\Zsq\Wsq}(\mid-T^2)^2=0, \\
    K_4 &=
    \hd(\nabla_1E_1,E_4)+\thalf\d_4\paraa{\delta\Zsq(\mid-T^2)^2}
    -\hd(E_1,E_1)2T\\
    &=-(\alpha_4+T)\delta\Zsq(\mid-T^2)^2
    +\paraa{\alpha_4+3T}\delta\Zsq(\mid-T^2)^2-2\delta\Zsq
    T(\mid-T^2)^2\\
    &=0.
  \end{align*}
  This shows that $\nabla_1E_1$ satisfies Kozul's formula \eqref{eq:Kozul.Uab}. The other
  connection components can be checked in an analogous way.

  Let us now consider the claim that
  \begin{align*}
    \nabla_{\dn_i}E_a=\paraa{\nabla_iE_a}T^n.
  \end{align*}
  This fact is easily derived from Kozul's formula. Namely, one notes
  that
  \begin{align*}
    \varphi\paraa{[\d_a,\dn[n]_i]}=\varphi\paraa{[\d_a,\d_i]}T^n+E_i(\d_aT^n)
  \end{align*}
  and computes using Kozul's formula:
  \begin{align*}
    2\hd\paraa{&\nablasub{\dn[n]_i}E_b,E_c}
    =\paraa{\d_i\hd(E_b,E_c)}T^n+\d_b\paraa{\hd(E_c,E_iT^n)}-\d_c\paraa{\hd(E_iT^n,E_b)}\\
    &\qquad -\hd\paraa{E_i,\varphi([\d_b,\d_c])}T^n
    +\hd\paraa{E_b,\varphi([\d_c,\dn[n]_i])}
    +\hd\paraa{E_c,\varphi([\dn[n]_i,\d_b])}\\
    &=\paraa{\d_i\hd_{bc}+\d_b\hd_{ci}-\d_c\hd_{ib}}T^n
    +\hd_{ci}(\d_bT^n)-\hd_{ib}(\d_cT^n)-\hd\paraa{E_i,\varphi([\d_b,\d_c])}T^n\\
    &\qquad+\hd\paraa{E_b,\varphi([\d_c,\d_i])}T^n+\hd_{bi}(\d_cT^n)
    +\hd\paraa{E_c,\varphi([\d_i,\d_b])}T^n-\hd_{ci}(\d_bT^n)\\
    &=2\hd\paraa{\nablasub{\d_i}E_b,E_c}T^n=2\hd\paraa{(\nabla_{\d_i}E_b)T^n,E_c},
  \end{align*}
  using that $\hd_{ab}=\hd_{ba}$ and the fact that $T$ is hermitian
  and central. Since the metric is non-degenerate, it follows that
  \begin{equation*}
    \nablasub{\dn[n]_i}E_b=(\nabla_{\d_i}E_b)T^n.
    \qedhere
  \end{equation*}
\end{proof}

\noindent Note that if $\alpha_1=\alpha_2=\alpha_3=0$, the connection
in Proposition~\ref{prop:4sphere.connection} only involves elements of
$\Sfourt$ and is therefore a valid connection for
$(M,\hd,\gphi,\nabla)$ over $\Sfourt$. In particular, this is true for
the unperturbed metric; i.e. for $\delta=\mid$.

In Section~\ref{sec:local.algebra} we constructed the projective
module $\TSfourt$ and showed that it is isomorphic to $M$ (as a right
$\Sfourtloc$-module) in
Proposition~\ref{prop:local.global.isomorphic}. As is well known, a
projective module defined by a projector $\P$, admits a connection
of the form
\begin{align*}
  \nablab_{\d}U = \P\paraa{e_i\d(U^i)}
\end{align*}
which is compatible with the canonical metric on the free
module. Thus, having argued that one may regard the module $M$ as a
localization of the (global) module $\TSfourt$, it is natural to ask
if the connection on $\TSfourt$, defined in the above manner,
coincides with the connection found in
Proposition~\ref{prop:4sphere.connection} for the unperturbed metric.

\begin{proposition}\label{prop:connections.coincide}
  Let $U=e_iU^i$ be an element of $\TSfourt=\P((\Sfourt)^5)$  (as
  defined in \eqref{eq:P.def}) and set
  \begin{align*}
    \nablab_{a}U=\P\paraa{e_i\d_a(U^i)},
  \end{align*}
  for $a=1,2,3,4$. Then $\nablab_aE_b=\nabla_aE_b$ for $a,b=1,2,3,4$
  and $\delta=\mid$. 
\end{proposition}

\begin{proof}
  Let us prove the statement by computing $\nablab_aE_b$ for
  $a,b=1,2,3,4$ (i.e. 16 components in total) and compare it with
  Proposition~\ref{prop:4sphere.connection} for $\delta=\mid$. Since
  the calculations are straight-forward we shall only present one of
  them here to illustrate how they are performed. Thus,
  \begin{align*}
    \nablab_1E_1 &=
    \P\paraa{\d_1(-X^2(\mid-T^2),X^1(\mid-T^2),0,0,0)}\\
    &=\P\paraa{(-X^1,-X^2,0,0,0)}(\mid-T^2)^2\\
    &=(-X^1,-X^2,0,0,0)(\mid-T^2)^2-e_iX^i\paraa{-(X^1)^2-(X^2)^2}(\mid-T^2)^2\\
    &=(-X^1,-X^2,0,0,0)(\mid-T^2)^2+(X^1,X^2,X^3,X^4,T)\Zsq(\mid-T^2)^2\\
    &=\paraa{X^1(\Zsq-\mid),X^2(\Zsq-\mid),X^3\Zsq,X^4\Zsq,T}(\mid-T^2)^2.
  \end{align*}
  Now, for comparison, we find $\nabla_1E_1$ from
  Proposition~\ref{prop:4sphere.connection} when $\delta=\mid$:
  \begin{align*}
    \nabla_1E_1 &=
    -E_3(\mid-T^2)-E_4T\Zsq(\mid-T^2)\\
    &=-(X^1\Wsq,X^2\Wsq,-X^3\Zsq,-X^4\Zsq,0)(\mid-T^2)\\
    &\qquad-(X^1T,X^2T,X^3T,X^4T,T^2-\mid)T\Zsq(\mid-T^2)\\
    &=\Big/\Wsq+T^2\Zsq=\mid-\Zsq-T^2+T^2\Zsq=(\mid-T^2)(\mid-\Zsq)\Big/\\
    &=-\paraa{X^1(\mid-\Zsq),X^2(\mid-\Zsq),-X^3\Zsq,-X^4\Zsq,-T}(\mid-T^2)^2
  \end{align*}
  which equals $\nablab_1E_1$. The remaining computations are done in
  an analogous way.
\end{proof}

\section{The Gauss-Chern-Bonnet theorem}\label{sec:gcb}

\subsection{The trace}\label{sec:trace}

\noindent
Just as for the noncommutative torus, one may introduce a linear
functional on $\Sfourt$ corresponding to integration on the classical
manifold. Namely, for a given basis
element $e^I$ with $I=(j,k,l,m,\epsilon)$ (in the notation of
Section~\ref{sec:basic.properties}) one defines a linear map
$\phi:\Sfourt\to C^{\infty}(S^4)$ via
\begin{align*}
  \phi(e^I) = e^{i(j-k)\xi_1}\paraa{\cos\varphi\cos\psi}^{j+k}
  e^{i(l-m)\xi_2}\paraa{\sin\varphi\cos\psi}^{l+m}(\sin\psi)^\epsilon
\end{align*}
and
\begin{align*}
  \tau(e^I)=
  \int_0^{2\pi}d\xi_1\int_0^{2\pi}d\xi_2\int_{-\pi/2}^{\pi/2}d\psi
  \int_0^{\pi/2}d\varphi\,\phi(e^I)\sin\varphi\cos\varphi\cos^3\psi,
\end{align*}
which are extended to $\Sfourt$ as linear maps
(cp. \cite{s:dynamical.spheres} for a similar approach in the
unperturbed case). The volume
element of the round metric $g_0$ on $S^4$ is given by
$\sin\varphi\cos\varphi\cos^3\psi\,d\xi_1d\xi_2 d\psi d\varphi$ and
for the perturbed metric $\delta g_0$ one obtains
\begin{align*}
  dV = \delta^2\sin\varphi\cos\varphi\cos^3\psi\,d\xi_1d\xi_2 d\psi d\varphi.
\end{align*}
In order to reflect the fact that one would like to integrate with respect to the
perturbed metric, we introduce
\begin{align*}
  \td(a) = \tau\paraa{\delta a\delta}.
\end{align*}
Let us note a few properties of the linear functional $\td$.  We start
with the following lemma:

\begin{lemma}\label{lemma:eI.zero.trace}
  Assume that $\theta\notin\rationals$ and $\delta\in\ZSfourt$. If $e^I\notin\ZSfourt$ then
  $\td(e^I)=0$.
\end{lemma}

\begin{proof}
  Let us start by considering $\td(e^I)$ when $I=(j,k,l,m,0)$.
  Assuming that $\delta\in\ZSfourt$ and $\theta\notin\rationals$, one
  may write 
  \begin{align*}
    \delta^2 = \sum_{i_1i_2\epsilon}a_{i_1i_2\epsilon}(\Zsq)^{i_1}\Wsq)^{i_2}T^\epsilon
  \end{align*}
  by Proposition~\ref{prop:algebra.center}, and
  \begin{align*}
    \td(e^I) &= \sum_{i_1i_2\epsilon}a_{i_1i_2\epsilon}
    \tau\paraa{e^{(j,k,l,m,0)}(\Zsq)^{i_1}(\Wsq)^{i_2}T^\epsilon}\\
    &=\sum_{i_1i_2\epsilon}a_{i_1i_2\epsilon}
    \tau\paraa{e^{(j+i_1,k+i_1,l+i_2,m+i_2,\epsilon)}}.
  \end{align*}
  Since 
  \begin{align*}
    \int_0^{2\pi}d\xi_1\int_0^{2\pi}d\xi_2e^{ik_1\xi_1}e^{ik_2\xi_2}=
    \begin{cases}
      4\pi^2\text{ if }k_1=k_2=0,\\
      0 \text{ otherwise},
    \end{cases}
  \end{align*}
  we conclude that $\td(e^{(j,k,l,m,0)})=0$ if $j\neq k$ or $l\neq m$,
  which is equivalent to $e^{(j,k,l,m,0)}\notin\ZSfourt$. Similarly,
  for $I=(j,k,l,m,1)$, terms proportional to $a_{i_1i_2 1}$ are of the
  form
  \begin{align*}
    a_{i_1i_2 1}\tau&\parab{e^{(j+i_1,k+i_1,l+i_2,m+i_2,0)}
      -e^{(j+i_1+1,k+i_1+1,l+i_2,m+i_2,0)}\\
      &\qquad-e^{(j+i_1,k+i_1,l+i_2+1,m+i_2+1,0)}}
  \end{align*}
  which, by using the same argument as above, implies that
  $\td(e^{(j,k,l,m,1)})=0$ if $j\neq k$ or $l\neq m$.
\end{proof}

\begin{proposition}
  If $\delta\in\ZSfourt$ and $\theta\notin\rationals$, then $\td$
  satisfies
  \begin{enumerate}
  \item $\td([a,b])=0$,
  \item $\td(a^\ast)=\overline{\td(a)}$,
  \end{enumerate}
  for all $a,b\in\Sfourt$.
\end{proposition}

\begin{proof}
  To prove (1), we show that $\td([e^{I_1},e^{I_2}])=0$. By using
  Lemma~\ref{lemma:basis.product} one obtains
  \begin{align*}
    \td\paraa{[e^{I_1},e^{I_2}]} &=
    \parab{q^{(l_1-m_1)(j_2-k_2)}-q^{(l_2-m_2)(j_1-k_1)}}
    \td\paraa{e^{I_1+I_2}}
  \end{align*}
  if $\epsilon_1+\epsilon_2\leq 1$, and
  \begin{equation}\label{eq:td.com}
    \begin{split}
      \td([e^{I_1},e^{I_2}]) &=
      \parab{q^{(l_1-m_1)(j_2-k_2)}-q^{(l_2-m_2)(j_1-k_1)}}\times\\
      &\quad
      \paraa{e^{(\Ih_1+\Ih_2,0)}-e^{(\Ih_1+\Ih_2+\oneh_Z,0)}-e^{(\Ih_1+\Ih_2+\oneh_W,0)}}
    \end{split}
  \end{equation}
  if $\epsilon_1+\epsilon_2=2$. From Lemma~\ref{lemma:eI.zero.trace} it
  follows that if $j_1+j_2\neq k_1+k_2$ or $l_1+l_2\neq m_1+m_2$ then
  $\td([e^{I_1},e^{I_2}])=0$. On the other hand, if $j_1+j_2=k_1+k_2$
  and $l_1+l_2=m_1+m_2$ then
  \begin{align*}
    (l_1-m_1)(j_2-k_2)=(l_2-m_2)(j_1-k_1)
  \end{align*}
  which gives $\td([e^{I_1},e^{I_2}])=0$ from \eqref{eq:td.com}.

  For (2), we again consider $a=\sum_Ia_Ie^I$ and find
  \begin{align*}
    \td(a^\ast) = \sum_I\overline{a_I}\td\paraa{(e^I)^\ast}
    =\sum_Iq^{(j-k)(l-m)}\overline{a_I}\td(e^I).
  \end{align*}
  Since $\td(e^I)=0$ if $j\neq k$ or $l\neq m$ (by
  Lemma~\ref{lemma:eI.zero.trace}), the above sum equals
  \begin{align*}
    \td(a^\ast) = \sum_I\overline{a_I}\td(e^I)=\overline{\td(a)}
  \end{align*}
  using that $\td(e^I)\in\reals$ when $j=k$ and $l=m$.
\end{proof}

\noindent For the forthcoming discussion of the Gauss-Chern-Bonnet
theorem, we extend $\td$ to the commutative subalgebra
$\Zloc\subseteq\Sfourtloc$ given by
\begin{align*}
  \Zloc = \complex\angles{\mid,\Zsq,\Zsqi,\Wsq,\Wsqi,
  T,(\mid-T^2)^{-1},(\mid+T^2)^{-1}},
\end{align*}
by defining a homomorphism (of commutative $\ast$-algebras)
$\phi_0:\Zloc\to C^\infty(U_0)$ as
\begin{alignat*}{2}
  &\phi_0(\Zsq) = \cos^2(\varphi)\cos^2(\psi) &\qquad &\phi_0(\Wsq) = \sin^2(\varphi)\cos^2(\psi)\\
  &\phi_0(\mid) = 1 & &\phi_0(T) = \sin(\psi)
\end{alignat*}
as well as
\begin{align*}
  &\phi_0\paraa{(\mid-T^2)^{-1}}=\frac{1}{\cos^2(\psi)}
    =\frac{1}{\phi_0(\mid-T^2)}\\
  &\phi_0\paraa{(\mid+T^2)^{-1}}=\frac{1}{1+\sin^2(\psi)}
    =\frac{1}{\phi_0(\mid+T^2)}\\
  &\phi_0\paraa{\Zsqi}=\frac{1}{\cos^2(\varphi)\cos^2(\psi)}
    =\frac{1}{\phi_0(\Zsq)}\\
  &\phi_0\paraa{\Wsqi}=\frac{1}{\sin^2(\varphi)\cos^2(\psi)}
    =\frac{1}{\phi_0(\Wsq)}.
\end{align*}
For $\phi_0$ to be well-defined, one needs to check that the above
definition is compatible with the relations in $\Zloc$. The only
nontrivial relation to check is
\begin{align*}
  \phi_0(\Zsq+\Wsq+T^2-\mid) &= \cos^2(\varphi)\cos^2(\psi)
  +\sin^2(\varphi)\cos^2(\psi)+\sin^2(\psi)-1\\
  &=\cos^2(\psi)+\sin^2(\psi)-1 = 0,
\end{align*}
which shows that $\phi_0$ is indeed well-defined. Note that $\phi_0$
coincides with $\phi$ on $\ZSfourt$. Finally, for $\delta\in\Zloc$,
we define 
\begin{align*}
  \tdloc(a) = 
  \int_0^{2\pi}d\xi_1\int_0^{2\pi}d\xi_2\int_{-\pi/2}^{\pi/2}d\psi
  \int_0^{\pi/2}d\varphi\,\phi_0(a)\phi_0(\delta^2)\cos^3\psi\sin\varphi\cos\varphi,
\end{align*}
for $a\in\Zloc$, whenever the above integral is convergent. (For
instance, the integral does not exists when $a=(\mid-T^2)^{-2}$.)

\subsection{The Gauss-Chern-Bonnet theorem}\label{sec:sub.gcb}

\noindent 
For a closed surface $\Sigma$, the Gauss-Bonnet theorem states that
the integral of the Gaussian curvature over $\Sigma$ is proportional
to the Euler characteristic of $\Sigma$. This provides an important
link between topology and Riemannian geometry. In particular, since
the Euler characteristic is independent of any metric tensor, the
integral gives the same value if we perturb the metric. This
theorem has been generalized to closed even dimensional Riemannian
manifolds, where the scalar curvature is replaced by the Pfaffian of
the curvature form.  In case of a closed four dimensional manifold
$M$, the Gauss-Chern-Bonnet theorem states that
\begin{align}\label{eq:gcb.formula}
  \chi(M) = \frac{1}{32\pi^2}\int_{M}\paraa{R^{abcd}R_{abcd}-4\Ric_{ab}\Ric^{ab}+S^2}d\mu
\end{align}
where $R_{abcd}$ is the Riemann curvature tensor, $\Ric_{ab}$ is the
Ricci curvature, $S$ denotes the scalar curvature and $\chi(M)$ is the
Euler characteristic of $M$. (Recall that $\chi(S^4)=2$.)  In this
section, we will show that there exists an analogue of the
Gauss-Chern-Bonnet theorem for the pseudo-Riemannian calculus of
$\Sfourt$ we have developed. Our approach is based on the fact that
all coefficients of the curvature tensor lie in the commutative
subalgebra $\Zloc$, which allows us to compute directly the
  Pfaffian of the curvature form.

Let us consider a metric perturbation $\delta\in\Zloc$ that is a
polynomial in $T$, and such that $\delta$ is invertible in $\Zloc$.
It follows that $\alpha_1=\alpha_2=\alpha_3=0$ (in the notation of
Section~\ref{sec:pseudo.Riemannian}), since
$\d_1T=\d_2T=\d_3T=0$. Moreover,
\begin{align*}
  \d_4\delta = \delta'(T)(\d_4T)\delta^{-1}\delta
  =-\delta'(T)(\mid-T^2)\delta^{-1}\delta
\end{align*}
where $\delta'(T)$ denotes the (formal) derivative of the polynomial
$\delta(T)$ with respect to $T$, which implies that
\begin{align*}
  \alpha\equiv\alpha_4 = -\half(\mid-T^2)\delta'\delta^{-1}.
\end{align*}
An example of such a perturbation is given by $\delta=(\mid+T^2)^N$ which
gives 
\begin{align*}
 \alpha = -NT(\mid-T^2)(\mid+T^2)^{-1}. 
\end{align*}
Moreover, by $\alpha'$ we shall denote the (formal) derivative of
$\alpha(T)$ with respect to $T$.  For easy reference, let us recall
the formulas from Proposition~\ref{prop:4sphere.connection} in the
situation where $\alpha_1=\alpha_2=\alpha_3=0$:
\begin{align*}
  \nabla_1E_1 &= -E_3(\mid-T^2)-E_4(\alpha+T)\Zsq(\mid-T^2) \\
  \nabla_2E_2 &= E_3(\mid-T^2)-E_4(\alpha+T)\Wsq(\mid-T^2)\\
  \nabla_3E_3 &=E_3(\Wsq-\Zsq)-E_4(\alpha+T)\Zsq\Wsq\\
  \nabla_4E_4 &= E_4(\alpha+T)
\end{align*}
\begin{alignat*}{2}
  \nabla_1E_2&=\nabla_2E_1=0 &\qquad
  \nabla_1E_3&=\nabla_3E_1=E_1\Wsq\\ 
  \nabla_1E_4 &= E_1(\alpha+T) &
  \nabla_4E_1 &= E_1(\alpha+3T)\\
  \nabla_2E_3 &=-E_2\Zsq & \nabla_3E_2&=-E_2\Zsq\\
  \nabla_2E_4 &= E_2(\alpha+T) &
  \nabla_4E_2 &= E_2(\alpha+3T) \\
  \nabla_3E_4 &= E_3(\alpha+T) &
  \nabla_4E_3 &= E_3(\alpha+3T).
\end{alignat*}
It is now straight-forward to compute the curvature:
\begin{align*}
  R(\d_1,\d_2)E_1 &= -E_2\paraa{\mid-(\alpha+T)^2}\Zsq(\mid-T^2)\\
  R(\d_1,\d_2)E_2 &= E_1\paraa{\mid-(\alpha+T)^2}\Wsq(\mid-T^2)\\
  R(\d_1,\d_2)E_3 &= 0\qquad\quad R(\d_1,\d_2)E_4 = 0\\
  R(\d_1,\d_3)E_1 &= -E_3\paraa{\mid-(\alpha+T)^2}\Zsq(\mid-T^2)\\
  R(\d_1,\d_3)E_3 &= E_1\paraa{\mid-(\alpha+T)^2}\Zsq\Wsq\\
  R(\d_1,\d_3)E_2 &= 0\qquad\quad R(\d_1,\d_3)E_4 = 0\\
  R(\d_1,\d_4)E_1 &= -E_4\paraa{\mid+\alpha'}\Zsq(\mid-T^2)^2\\
  R(\d_1,\d_4)E_4 &= E_1(\mid+\alpha')(\mid-T^2)\\
  R(\d_1,\d_4)E_2 &= 0\qquad\quad R(\d_1,\d_4)E_3 = 0\\  
  R(\d_2,\d_3)E_2 &= -E_3\paraa{\mid-(\alpha+T)^2}\Wsq(\mid-T^2)\\
  R(\d_2,\d_3)E_3 &= E_2\paraa{\mid-(\alpha+T)^2}\Zsq\Wsq\\
  R(\d_2,\d_3)E_1 &= 0\qquad\quad R(\d_2,\d_3)E_4 = 0\\
  R(\d_2,\d_4)E_2 &= -E_4(\mid+\alpha')\Wsq(\mid-T^2)^2\\
  R(\d_2,\d_4)E_4 &= E_2(\mid+\alpha')(\mid-T^2)\\
  R(\d_2,\d_4)E_1 &= 0\qquad\quad R(\d_2,\d_4)E_3 = 0\\
  R(\d_3,\d_4)E_3 &= -E_4(\mid+\alpha')\Zsq\Wsq(\mid-T^2)\\
  R(\d_3,\d_4)E_4 &= E_3(\mid+\alpha')(\mid-T^2)\\
  R(\d_3,\d_4)E_1 &= 0\qquad\quad R(\d_3,\d_4)E_2 = 0
\end{align*}
and the only non-zero curvature components
$R_{abpq}=\hd(E_a,R(\d_p,\d_q)E_b)$ turn out to be
\begin{align*}
  R_{1212} &= \delta\paraa{\mid-(\alpha+T)^2}\Zsq\Wsq(\mid-T^2)^3\\
  R_{1313} &= \delta\paraa{\mid-(\alpha+T)^2}|Z|^4\Wsq(\mid-T^2)^2\\
  R_{1414} &= \delta(\mid+\alpha')\Zsq(\mid-T^2)^3\\
  R_{2323} &= \delta\paraa{\mid-(\alpha+T)^2}\Zsq|W|^4(\mid-T^2)^2\\
  R_{2424} &= \delta(\mid+\alpha')\Wsq(\mid-T^2)^3\\
  R_{3434} &= \delta(\mid+\alpha')\Zsq\Wsq(\mid-T^2)^2.
\end{align*}
In the local algebra $\Zloc$, the metric $\hd$ is invertible since $\delta$ is
invertible. Moreover, every component of the metric, as well as of the
curvature, is central, which implies that there exists a naive analogue
of the integrand in \eqref{eq:gcb.formula}. Setting 
\begin{align*}
  R^{abcd} &= (\hd)^{ap}(\hd)^{bq}(\hd)^{cr}(\hd)^{ds}R_{pqrs}\\
  \Ric_{ab} &= (\hd)^{pq}R_{apbq}\\
  \Ric^{ab} &= (\hd)^{ap}(\hd)^{bq}\Ric_{pq}\\
  S &= (\hd)^{ab}\Ric_{ab}
\end{align*}
one finds that
\begin{align}\label{eq:gcb.term}
  R^{abcd}R_{abcd}&-4\Ric_{ab}\Ric^{ab}+S^2
  =24\paraa{\mid-(\alpha+T)^2}(\mid+\alpha')(\mid-T^2)^{-1}\delta^{-2}.
\end{align}

\begin{theorem}\label{thm:gcb.theorem}
  Let $\delta(T)$ be an invertible polynomial in $\Zloc$ and define
  $\alpha$ via the relation $\d_4\delta=2\alpha\delta$. If
  \begin{align*}
    \phi_0(\alpha)\big|_{\psi=\frac{\pi}{2}}=\phi_0(\alpha)\big|_{\psi=-\frac{\pi}{2}} = 0,
  \end{align*}
  then
  \begin{align*}
    \chi(\Sfourt)=\frac{1}{32\pi^2}\tdloc\paraa{R^{abcd}R_{abcd}-4\Ric_{ab}\Ric^{ab}+S^2}=2.
  \end{align*}
\end{theorem}

\begin{proof}
  Since $\delta$ is a polynomial in $T$ and $\d_4T=T^2-\mid$, one can
  express $\alpha$ in terms of $T$ and, by a slight abuse of notation,
  we let $\alpha(t)$ be such that
  $\phi_0(\alpha)=\alpha(\sin\psi)$. In this notation, the assumption
  on $\phi_0(\alpha)$ may be stated as $\alpha(1)=\alpha(-1)=0$.

  From the definition of $\tdloc$ it follows that
  \begin{align*}
    \chi&=\frac{1}{32\pi^2}\tdloc\paraa{R^{abcd}R_{abcd}-4\Ric_{ab}\Ric^{ab}+S^2}\\
    &= I_\psi\int_{0}^{2\pi}d\xi_1\int_0^{2\pi}d\xi_2\int_0^{\frac{\pi}{2}}
    \sin\varphi\cos\varphi d\varphi,
  \end{align*}
  where
  \begin{align*}
    I_\psi 
    &=\frac{24}{32\pi^2}\int_{-\frac{\pi}{2}}^{\frac{\pi}{2}}
    \paraa{1-(\alpha(\sin\psi)+\sin\psi)^2}(1+\alpha'(\sin\psi))\cos\psi
    d\psi.
  \end{align*}
  Substituting $t=\sin\psi$ gives
  \begin{align*}
    I_\psi = \frac{24}{32\pi^2}\int_{-1}^1\paraa{1-(\alpha(t)+t)^2}(1+\alpha'(t))dt,
  \end{align*}
  which can easily be integrated to
  \begin{align*}
    I_\psi =
    \frac{24}{32\pi^2}\bracketb{\alpha(t)+t-\tfrac{1}{3}\paraa{\alpha(t)+t}^3}_{-1}^1
    =\frac{24}{32\pi^2}\parab{1-\tfrac{1}{3}+1-\tfrac{1}{3}}=\frac{1}{\pi^2},
  \end{align*}
  since $\alpha(1)=\alpha(-1)=0$. Finally, one obtains
  \begin{align*}
    \chi=I_\psi&\int_{0}^{2\pi}d\xi_1\int_0^{2\pi}d\xi_2\int_0^{\frac{\pi}{2}}
    \sin\varphi\cos\varphi d\varphi\\
    &=\frac{1}{\pi^2}\int_{0}^{2\pi}d\xi_1\int_0^{2\pi}d\xi_2\int_0^{\frac{\pi}{2}}
    \sin\varphi\cos\varphi d\varphi=\frac{1}{\pi^2}\cdot 4\pi^2\cdot\frac{1}{2}=2,
  \end{align*}
  which proves the statement.
\end{proof}

\noindent
In this paper, we have preferred to stay in the purely algebraic
regime, and have thus not considered any smooth completion of
$\Sfourt$, in order to stress the point that our results do not depend
on the analytic structure. However, we expect that
Theorem~\ref{thm:gcb.theorem} holds true even for more general
perturbations in a potentially larger algebra. For instance, if
$\delta=e^{\lambda T}$ exists for all $\lambda\in\reals$, one obtains
$\alpha=\tfrac{\lambda}{2}(T^2-\mid)$ which clearly fulfills the
conditions of Theorem~\ref{thm:gcb.theorem}. Moreover, one may
consider perturbations given, not only as functions of $T$, but as
more general elements of $\Zloc$. Although our approach to the
Gauss-Bonnet-Theorem may be too naive to have any impact on the
general problem, we hope that our investigations will contribute to
the growing understanding of Riemannian curvature in noncommutative geometry.

\section*{Acknowledgment}

\noindent We would like to thank M. Khalkhali and
J. Rosenberg for discussions during the \textit{Fields Workshop on the
Geometry of Noncommutative Manifolds} in March, 2015.
Furthermore, J. A. is supported by the Swedish Research Council.

\bibliographystyle{alpha}
\bibliography{sphere_curvature}  

\end{document}